\documentclass[11pt]{article}

\usepackage{amsmath, amssymb, amsthm, calc}  

\title{How to split Dyson's transference inequality \\ with the help of \\ Wolfgang Schmidt's parametric geometry of numbers.
                              \thanks{ This research was  supported by
                              RFBR (grant $\textup N^{\circ}$ 09--01--00371a) and
                              by the grant of the President of Russian Federation
                              $\textup N^\circ$ MK--1226.2010.1.
                             }}
\author{Oleg\,N.\,German}
\date{}

\theoremstyle{definition}
\newtheorem{definition}{Definition}

\theoremstyle{remark}

\theoremstyle{plain}
\newtheorem{theorem}{Theorem}

\newtheorem{proposition}{Proposition}
\newtheorem{corollary}{Corollary}

\newtheorem{classic}{Theorem}
\newtheorem{classicprime}[classic]{Theorem}

\DeclareMathOperator{\spanned}{span}

\renewcommand{\vec}[1]{\mathbf{#1}}
\renewcommand{\geq}{\geqslant}
\renewcommand{\leq}{\leqslant}
\renewcommand{\phi}{\varphi}

\newcommand{\R}{\mathbb{R}}
\newcommand{\Z}{\mathbb{Z}}

\newcommand{\La}{\Lambda}

\newcommand{\bpsi}{\underline{\psi}}
\newcommand{\apsi}{\overline{\psi}}
\newcommand{\bPsi}{\underline{\Psi}}
\newcommand{\aPsi}{\overline{\Psi}}

\newcommand{\cL}{\mathcal{L}}

\newcommand{\cB}{\mathcal{B}}

\newcommand{\cJ}{\mathcal{J}}

\newcommand{\gT}{\mathfrak{T}}
\newcommand{\ga}{\mathfrak{a}}
\newcommand{\gb}{\mathfrak{b}}

\newcommand{\tr}[1]{{#1}^\intercal}

\begin{document}

  \maketitle

  \begin{abstract}
    In this paper we develop some of the ideas belonging to W.~Schmidt and L.~Summerer to define intermediate Diophantine exponents and split Dyson's transference inequality into a chain of inequalities for intermediate exponents. This splitting generalizes the analogous result of M.~Laurent and Y.~Bugeaud for Khintchine's transference inequalities.
  \end{abstract}

  \section{Introduction}

  Consider a system of linear equations
  \begin{equation} \label{eq:the_system}
    \Theta\vec x=\vec y
  \end{equation}
  with $\vec x\in\R^m$, $\vec y\in\R^n$ and
  \[ \Theta=
     \begin{pmatrix}
       \theta_{11} & \cdots & \theta_{1m} \\
       \vdots & \ddots & \vdots \\
       \theta_{n1} & \cdots & \theta_{nm}
     \end{pmatrix},\qquad
     \theta_{ij}\in\R. \]
  The classical measure of how well the space of solutions to this system can be approximated by integer points is defined as follows. Let $|\cdot|$ denote the sup-norm in the corresponding space.

  \begin{definition} \label{def:belpha_1}
    The supremum of the real numbers $\gamma$, such that there are arbitrarily large values of $t$ for which (resp. such that for every $t$ large enough) the system of inequalities
    \begin{equation} \label{eq:belpha_1_definition}
      |\vec x|\leq t,\qquad|\Theta\vec x-\vec y|\leq t^{-\gamma}
    \end{equation}
    has a nonzero solution in $(\vec x,\vec y)\in\Z^m\oplus\Z^n$, is called the \emph{regular} (resp. \emph{uniform}) \emph{Diophantine exponent} of $\Theta$ and is denoted by $\beta_1$ (resp. $\alpha_1$).
  \end{definition}

  This paper is a result of the attempt to generalize this concept to the case of the problem of approximating the space of solutions to \eqref{eq:the_system} by $p$-dimensional rational subspaces of $\R^{m+n}$. A large work in this direction was made by W.~Schmidt in \cite{schmidt_annals}. Later, in \cite{laurent_up_down}, \cite{bugeaud_laurent_up_down}, a corresponding definition was given by M.~Laurent and Y.~Bugeaud in the case when $m=1$. With their definition they were able to split the classical Khintchine transference principle into a chain of inequalities for intermediate exponents. However, the way we defined $\alpha_1$ and $\beta_1$ naturally proposes a generalization, which appears to be different from Laurent's:

  \begin{definition} \label{def:belpha_p}
    The supremum of the real numbers $\gamma$, such that there are arbitrarily large values of $t$ for which (resp. such that for every $t$ large enough) the system of inequalities
    \begin{equation} \label{eq:belpha_p_definition}
      |\vec x|\leq t,\qquad|\Theta\vec x-\vec y|\leq t^{-\gamma}
    \end{equation}
    has $p$ solutions $\vec z_i=(\vec x_i,\vec y_i)\in\Z^m\oplus\Z^n$, $i=1,\ldots,p$, linearly independent over $\Z$, is called the \emph{$p$-th regular (resp. uniform) Diophantine exponent of the first type} of $\Theta$ and is denoted by $\beta_p$ (resp. $\alpha_p$).
  \end{definition}

  In Section \ref{sec:laurexp} we propose a definition of intermediate exponents of the second type, which is consistent with Laurent's. In subsequent Sections we show the connection between these two generalizations and some exponents that naturally emerge in Schmidt's parametric geometry of numbers developed in \cite{schmidt_summerer}. Then we discuss the properties of these quantities, generalize some of the observations made in \cite{schmidt_summerer}, and split Dyson's transfer inequality into a chain of inequalities for the intermediate exponents of the second type.

  \section{Laurent's exponents and their generalization} \label{sec:laurexp}

  Set $d=m+n$. Let us denote by $\pmb\ell_1,\ldots,\pmb\ell_d$ the columns of the matrix
  \[ \begin{pmatrix}
       E_m & -\tr\Theta\\
       \Theta & E_n
     \end{pmatrix}, \]
  where $E_m$ and $E_n$ are the corresponding unity matrices and $\tr\Theta$ is the transpose of $\Theta$. Clearly, $\cL=\spanned_\R(\pmb\ell_1,\ldots,\pmb\ell_m)$ is the space of solutions to the system \eqref{eq:the_system}, and $\cL^\bot=\spanned_\R(\pmb\ell_{m+1},\ldots,\pmb\ell_d)$. Denote also by $\vec e_1,\ldots,\vec e_d$ the columns of the $d\times d$ unity matrix $E_d$.

  The following Definition is a slightly modified Laurent's one.

  \begin{definition} \label{def:ba_for_m_equal_to_1}
    Let $m=1$. The supremum of the real numbers $\gamma$, such that there are arbitrarily large values of $t$ for which (resp. such that for every $t$ large enough) the system of inequalities
    \begin{equation} \label{eq:ba_for_m_equal_to_1}
      |\vec Z|\leq t,\qquad|\pmb\ell_1\wedge\vec Z|\leq t^{-\gamma}
    \end{equation}
    has a nonzero solution in $\vec Z\in\wedge^p(\Z^d)$ is called the \emph{$p$-th regular (resp. uniform) Diophantine exponent of the second type} of $\Theta$ and is denoted by $\gb_p$ (resp. $\ga_p$).
  \end{definition}

  Here $\vec Z\in\wedge^p(\R^d)$, $\pmb\ell_1\wedge\vec Z\in\wedge^{p+1}(\R^d)$ and for each $q$ we consider $\wedge^q(\R^d)$ as a $\binom dq$-dimensional Euclidean space with the orthonormal basis consisting of the multivectors
  \[ \vec e_{i_1}\wedge\ldots\wedge\vec e_{i_q},\qquad 1\leq i_1<\ldots<i_q\leq d, \]
  and denote by $|\cdot|$ the sup-norm with respect to this basis.

  Laurent denoted the exponents $\gb_p$, $\ga_p$ as $\omega_{p-1}$, $\hat\omega_{p-1}$, respectively, and showed that for $p=1$ they coincide with $\beta_1$, $\alpha_1$. He also noticed that one does not have to require $\vec Z$ to be decomposable in Definition \ref{def:ba_for_m_equal_to_1}, which essentially simplifies working in $\wedge^p(\R^d)$.

  In order to generalize Definition \ref{def:ba_for_m_equal_to_1} let us set for each $\sigma=\{i_1,\ldots,i_k\}$, $1\leq i_1<\ldots<i_k\leq d$,
  \begin{equation} \label{eq:L_sigma}
    \vec L_\sigma=\pmb\ell_{i_1}\wedge\ldots\wedge\pmb\ell_{i_k},
  \end{equation}
  denote by $\cJ_k$ the set of all the $k$-element subsets of $\{1,\ldots,m\}$, $k=0,\ldots,m$, and set $\vec L_\varnothing=1$.

  Let us also set $k_0=\max(0,m-p)$.

  \begin{definition} \label{def:ba}
    The supremum of the real numbers $\gamma$, such that there are arbitrarily large values of $t$ for which (resp. such that for every $t$ large enough) the system of inequalities
    \begin{equation} \label{eq:ba}
      \max_{\sigma\in\cJ_k}|\vec L_\sigma\wedge\vec Z|\leq t^{1-(k-k_0)(1+\gamma)},\qquad k=0,\ldots,m,
    \end{equation}
    has a nonzero solution in $\vec Z\in\wedge^p(\Z^d)$ is called the \emph{$p$-th regular (resp. uniform) Diophantine exponent of the second type} of $\Theta$ and is denoted by $\gb_p$ (resp. $\ga_p$).
  \end{definition}

  We tended to make this definition look as simple as possible. However, it will be more convenient to work with in the multilinear algebra setting after it is slightly reformulated. To give the desired reformulation let us set for each $\sigma=\{i_1,\ldots,i_k\}$, $1\leq i_1<\ldots<i_k\leq d$,
  \begin{equation} \label{eq:E_sigma}
    \vec E_\sigma=\vec e_{i_1}\wedge\ldots\wedge\vec e_{i_k},
  \end{equation}
  denote by $\cJ'_k$ the set of all the $k$-element subsets of $\{m+1,\ldots,d\}$, $k=0,\ldots,n$, and set $\vec E_\varnothing=1$.

  Set also $k_1=\min(m,d-p)$.

  \begin{proposition} \label{prop:ba_substitution}
    The inequalities \eqref{eq:ba} can be substituted by
    \begin{equation} \label{eq:ba_modified}
      \max_{\begin{subarray}{c} \sigma\in\cJ_k \\ \sigma'\in\cJ'_{d-p-k} \end{subarray}}
      |\vec L_\sigma\wedge\vec E_{\sigma'}\wedge\vec Z|\leq t^{1-(k-k_0)(1+\gamma)},\qquad k=k_0,\ldots,k_1.
    \end{equation}
  \end{proposition}

  \begin{proof}
    Since $\pmb\ell_1,\ldots,\pmb\ell_m,\vec e_{m+1},\ldots,\vec e_d$ form a basis of $\R^d$, for each $q=1,\dots,d$ the multivectors
    \[ \vec L_\rho\wedge\vec E_{\rho'},\qquad\rho\in\cJ_j,\ \rho'\in\cJ'_{q-j},\ \max(0,q-n)\leq j\leq\min(q,m), \]
    form a basis of $\wedge^q(\R^d)$. Let us denote by $|\cdot|_\Theta$ the sup-norm in each $\wedge^q(\R^d)$ with respect to such a basis. Since any two norms in a Euclidean space are equivalent, and since in Definition \ref{def:ba} we are concerned only about exponents, we can substitute \eqref{eq:ba} by
    \begin{equation} \label{eq:ba_Thetanized}
      \max_{\sigma\in\cJ_k}|\vec L_\sigma\wedge\vec Z|_\Theta\leq t^{1-(k-k_0)(1+\gamma)},\qquad k=0,\ldots,m,
    \end{equation}
    and \eqref{eq:ba_modified} by
    \begin{equation} \label{eq:ba_modified_Thetanized}
      \max_{\begin{subarray}{c} \sigma\in\cJ_k \\ \sigma'\in\cJ'_{d-p-k} \end{subarray}}
      |\vec L_\sigma\wedge\vec E_{\sigma'}\wedge\vec Z|_\Theta\leq t^{1-(k-k_0)(1+\gamma)},\qquad k=k_0,\ldots,k_1.
    \end{equation}
    Writing
    \[ \vec Z=\sum_{j=\max(0,p-n)}^{\min(p,m)}\sum_{\begin{subarray}{c} \rho\in\cJ_j \\ \rho'\in\cJ'_{p-j} \end{subarray}}
       Z_{\rho,\rho'}\vec L_\rho\wedge\vec E_{\rho'}, \]
    we see that \eqref{eq:ba_Thetanized} for each $k$ means exactly that
    \[ \ Z_{\rho,\rho'}=0,\qquad\qquad\qquad\quad\ \text{ if }\rho\in\cJ_j,\ j>m-k, \]
    \begin{equation} \label{eq:coordinates_filtered}
      |Z_{\rho,\rho'}|\leq t^{1-(k-k_0)(1+\gamma)},\qquad\text{ if }\rho\in\cJ_j,\ j\leq m-k.
    \end{equation}
    Hence we see that all the inequalities in \eqref{eq:ba_Thetanized} with $k>k_1$ are trivial. Next, since we are concerned about large values of $t$, by Minkowski's first convex body theorem we may confine ourselves to considering only positive values of $1+\gamma$. Then the function $t^{1-(k-k_0)(1+\gamma)}$ is non-increasing with respect to $k$, so for each $\rho\in\cJ_j$ of all the inequalities \eqref{eq:coordinates_filtered} we may keep the ones with the largest $k$, i.e. with the one equal to $m-j$. Thus, \eqref{eq:ba_Thetanized} becomes equivalent to
    \begin{equation} \label{eq:coordinates_graduated}
      |Z_{\rho,\rho'}|\leq t^{1-(k-k_0)(1+\gamma)},\qquad\text{ if }\rho\in\cJ_{m-k},\ k_0\leq k\leq k_1.\phantom{,\ \rho'\in\cJ'_{p-m+k}}
    \end{equation}
    On the other hand, \eqref{eq:ba_modified_Thetanized} means that
    \begin{equation} \label{eq:coordinates_sandwiched}
      |Z_{\rho,\rho'}|\leq t^{1-(k-k_0)(1+\gamma)},\qquad\text{ if }\rho\in\cJ_{m-k},\ \rho'\in\cJ'_{p-m+k},\ k_0\leq k\leq k_1,
    \end{equation}
    which is obviously equivalent to \eqref{eq:coordinates_graduated}.
  \end{proof}

  \section{Schmidt's exponents}

  Let $\La$ be a unimodular $d$-dimensional lattice in $\R^d$. Denote by $\cB_\infty^d$ the unit ball in sup-norm, i.e. the cube with vertices at the points $(\pm1,\ldots,\pm1)$. For each vector $\pmb\tau=(\tau_1,\ldots,\tau_d)\in\R^d$ denote by $D_{\pmb\tau}$ the diagonal $d\times d$ matrix with $e^{\tau_1},\ldots,e^{\tau_d}$ on the main diagonal.
  Let us also denote by $\lambda_p(M)$ the $p$-th successive minimum of a compact symmetric convex body $M\subset\R^d$ (centered at the origin) with respect to the lattice $\La$.


  Suppose we have a path $\gT$ in $\R^d$ defined as $\pmb\tau=\pmb\tau(s)$, $s\in\R_+$, such that
  \begin{equation} \label{eq:sum_of_taus_is_zero}
    \tau_1(s)+\ldots+\tau_d(s)=0,\quad\text{ for all }s.
  \end{equation}
  In our further applications to Diophantine approximation we shall confine ourselves to a path that is a ray with the endpoint at the origin and all the functions $\tau_1(s),\ldots,\tau_d(s)$ being linear. However, in this Section, as well as in the next one, all the definitions and statements are given for arbitrary paths and lattices.

  Set $\cB(s)=D_{\pmb\tau(s)}\cB_\infty^d$. Consider the functions
  \[ \psi_p(\La,\gT,s)=\frac{\ln(\lambda_p(\cB(s)))}{s},\qquad p=1,\ldots,d. \]

  \begin{definition} \label{def:schmidt_psi}
    We call the quantities
    \[ \bpsi_p(\La,\gT)=\liminf_{s\to+\infty}\psi_p(\La,\gT,s),\qquad
       \apsi_p(\La,\gT)=\limsup_{s\to+\infty}\psi_p(\La,\gT,s) \]
    \emph{the $p$-th lower} and \emph{upper Schmidt's exponents of the first type}, respectively.
  \end{definition}

  \begin{definition} \label{def:schmidt_Psi}
    We call the quantities
    \[ \bPsi_p(\La,\gT)=\liminf_{s\to+\infty}\bigg(\sum_{i=1}^p\psi_i(\La,\gT,s)\bigg)\,,\qquad
       \aPsi_p(\La,\gT)=\limsup_{s\to+\infty}\bigg(\sum_{i=1}^p\psi_i(\La,\gT,s)\bigg) \]
    \emph{the $p$-th lower} and \emph{upper Schmidt's exponents of the second type}, respectively.
  \end{definition}

  Sometimes, when it is clear from the context what lattice and what path are under consideration, we shall write simply $\psi_p(s)$, $\bpsi_p$, $\apsi_p$, $\bPsi_p$, and $\aPsi_p$.

  The following Proposition and its Corollaries generalize some of the observations made in \cite{schmidt_summerer} and \cite{bugeaud_laurent_up_down}.

  \begin{proposition} \label{prop:mink}
    For any $\La$ and $\gT$ we have
    \begin{equation} \label{eq:mink}
      0\leq-\sum_{i=1}^d\psi_i(s)=O(s^{-1}).
    \end{equation}
    Particularly,
    \begin{equation} \label{eq:mink_0}
      \bPsi_d=\aPsi_d=\lim_{s\to\infty}\sum_{i=1}^d\psi_i(s)=0.
    \end{equation}
  \end{proposition}

  \begin{proof}
    Due to \eqref{eq:sum_of_taus_is_zero} the volumes of all the parallelepipeds $\cB(s)$ are equal to $2^d$, so by Minkowski's second theorem we have
    \[ \frac{1}{d!}\leq\prod_{i=1}^d\lambda_i(\cB(s))\leq1. \]
    Hence
    \[ -\frac{\ln(d!)}{s}\leq\sum_{i=1}^d\psi_i(s)\leq0, \]
    which immediately implies \eqref{eq:mink}.
  \end{proof}

  \begin{corollary} \label{cor:Psi_inter_dyson}
    For any $\La$ and $\gT$ and any $p$ within the range $1\leq p\leq d-2$ we have
    \begin{equation} \label{eq:Psi_inter_dyson}
      \frac{\bPsi_{p+1}}{d-p-1}\leq\frac{\bPsi_p}{d-p}\qquad\text{ and }\qquad\frac{\aPsi_{p+1}}{d-p-1}\leq\frac{\aPsi_p}{d-p}\,.
    \end{equation}
  \end{corollary}

  \begin{proof}
    Since $\psi_{p+1}(s)\leq\psi_{p+2}(s)\leq\ldots\leq\psi_d(s)$, it follows from \eqref{eq:mink} that
    \[ \psi_{p+1}(s)\leq\frac{-1}{d-p}\sum_{i=1}^p\psi_i(s), \]
    whence
    \[ \sum_{i=1}^{p+1}\psi_i(s)\leq\left(1-\frac{1}{d-p}\right)\sum_{i=1}^p\psi_i(s). \]
    It remains to take the $\liminf$ and the $\limsup$ of both sides as $s\to\infty$.
  \end{proof}

  Applying consequently \eqref{eq:Psi_inter_dyson} we get the following statement.

  \begin{corollary} \label{cor:Psi_dyson}
    For any $\La$ and $\gT$ we have
    \begin{equation} \label{eq:Psi_dyson}
      (d-1)\bPsi_{d-1}\leq\bPsi_1\qquad\text{ and }\qquad(d-1)\aPsi_{d-1}\leq\aPsi_1.
    \end{equation}
  \end{corollary}

  Another simple corollary to Proposition \ref{prop:mink} is the following statement.

  \begin{corollary} \label{cor:Psi_and_psi}
    For any $\La$ and $\gT$ we have
    \begin{equation} \label{eq:Psi_and_psi}
      \bPsi_{d-1}=-\apsi_d\qquad\text{ and }\qquad\aPsi_{d-1}=-\bpsi_d.
    \end{equation}
  \end{corollary}

  As we shall see later, the first of the inequalities \eqref{eq:Psi_dyson} generalizes Khintchine's and Dyson's transference inequalities.

  \section{Schmidt's exponents of the second type from the point of view of multilinear algebra}

  As before, let us consider the space $\wedge^p(\R^d)$ as the $\binom dp$-dimensional Euclidean space with the orthonormal basis consisting of the multivectors
  \[ \vec e_{i_1}\wedge\ldots\wedge\vec e_{i_p},\qquad 1\leq i_1<\ldots<i_p\leq d. \]
  Let us order the set of the $p$-element subsets of $\{1,\ldots,d\}$ lexicographically and denote the $j$-th subset by $\sigma_j$. To each vector $\pmb\tau=(\tau_1,\ldots,\tau_d)$ let us associate the vector
  \begin{equation} \label{eq:T_hat}
    \widehat{\pmb\tau}=\Big(\widehat\tau_1,\ldots,\widehat\tau_r\Big),\qquad\widehat\tau_j=\sum_{i\in\sigma_j}\tau_i,\qquad r={\binom dp}.
  \end{equation}
  Thus, a path $\gT:s\to\pmb\tau(s)$ leads us by \eqref{eq:T_hat} to the path $\widehat\gT:s\to\widehat{\pmb\tau}(s)$ also satisfying the condition
  \[ \widehat\tau_1(s)+\ldots+\widehat\tau_r(s)=0,\quad\text{ for all }s. \]
  Finally, given a lattice $\La\subset\R^d$, let us associate to it the lattice $\widehat\La=\wedge^p(\La)$.

  \begin{proposition} \label{prop:Psi_p_is_Psi_1}
    For any $\La$ and $\gT$ we have
    \[ \bPsi_p(\La,\gT)=\bPsi_1(\widehat\La,\widehat\gT)=\bpsi_1(\widehat\La,\widehat\gT)
    \quad\text{ and }\quad
    \aPsi_p(\La,\gT)=\aPsi_1(\widehat\La,\widehat\gT)=\apsi_1(\widehat\La,\widehat\gT). \]
  \end{proposition}

  \begin{proof}
    Let us denote by $\lambda_i(M)$ the $i$-th successive minimum of a body $M$ with respect to $\La$ if $M\subset\R^d$ and with respect to $\widehat\La$ if $M\in\wedge^p(\R^d)$.

    The matrix $D_{\widehat{\pmb\tau}}$ is the $p$-th compound of $D_{\pmb\tau}$:
    \[ D_{\widehat{\pmb\tau}}=D_{\pmb\tau}^{(p)}. \]
    This means that $D_{\widehat{\pmb\tau}}\cB_\infty^r$ is comparable to Mahler's $p$-th compound convex body of $D_{\pmb\tau}\cB_\infty^d$ (see \cite{mahler_compound_I}), i.e. there is a positive constant $c$ depending only on $d$, such that
    \begin{equation} \label{eq:comparable_to_pseudo_compound}
      c^{-1}D_{\widehat{\pmb\tau}}\cB_\infty^r\subset[D_{\pmb\tau}\cB_\infty^d]^{(p)}\subset cD_{\widehat{\pmb\tau}}\cB_\infty^r.
    \end{equation}
    In \cite{schmidt_DA} the set $D_{\widehat{\pmb\tau}}\cB_\infty^r$ is called the $p$-th pseudo-compound parallelepiped for $D_{\pmb\tau}\cB_\infty^d$.

    It follows from Mahler's theory of compound bodies that
    \begin{equation} \label{eq:first_minimum_vs_product}
      \lambda_1\left([D_{\pmb\tau}\cB_\infty^d]^{(p)}\right)\asymp\prod_{i=1}^p\lambda_i\left(D_{\pmb\tau}\cB_\infty^d\right)
    \end{equation}
    with the implied constants depending only on $d$. Combining \eqref{eq:comparable_to_pseudo_compound} and \eqref{eq:first_minimum_vs_product} we get
    \[ \ln\left(\lambda_1\left(D_{\widehat{\pmb\tau}(s)}\cB_\infty^r\right)\right)=
       \sum_{i=1}^p\ln\left(\lambda_i\left(D_{\pmb\tau(s)}\cB_\infty^d\right)\right)+O(1), \]
    whence
    \[ \psi_1(\widehat\La,\widehat\gT,s)=\sum_{i=1}^p\psi_i(\La,\gT,s)+o(1). \]
    It remains to take the $\liminf$ and the $\limsup$ of both sides as $s\to\infty$.
  \end{proof}

  \section{Diophantine exponents in terms of Schmidt's exponents}

  Let $\pmb\ell_1,\ldots,\pmb\ell_d$, $\vec e_1,\ldots,\vec e_d$ be as in Section \ref{sec:laurexp}. Set
  \[ T=
     \begin{pmatrix}
       E_m & 0 \\
       \Theta & E_n
     \end{pmatrix}. \]
  Then
  \[ \tr{(T^{-1})}=
     \begin{pmatrix}
       E_m & -\tr\Theta \\
       0 & E_n
     \end{pmatrix}, \]
  so the bases $\pmb\ell_1,\ldots,\pmb\ell_m,\vec e_{m+1},\ldots,\vec e_d$ and $\vec e_1,\ldots,\vec e_m,\pmb\ell_{m+1},\ldots,\pmb\ell_d$ are dual.

  Let us specify a lattice $\La$ and a path $\gT$ as follows. Set
  \begin{equation} \label{eq:La}
    \La=T^{-1}\Z^d=\Big\{ \Big(\langle\vec e_1,\vec z\rangle,\ldots,\langle\vec e_m,\vec z\rangle,\langle\pmb\ell_{m+1},\vec z\rangle,\ldots,\langle\pmb\ell_d,\vec z\rangle\Big)\in\R^d \,\Big|\, \vec z\in\Z^d \Big\}
  \end{equation}
  and define $\gT:s\mapsto\pmb\tau(s)$ by
  \begin{equation} \label{eq:path}
    \tau_1(s)=\ldots=\tau_m(s)=s,\quad\tau_{m+1}(s)=\ldots=\tau_d(s)=-ms/n.
  \end{equation}
  Schmidt's exponents $\bpsi_p$, $\apsi_p$ corresponding to such $\La$ and $\gT$ and the exponents $\beta_p$, $\alpha_p$ are but two different points of view at the same phenomenon. The same can be said about $\bPsi_p$, $\aPsi_p$ and $\gb_p$, $\ga_p$. It is exposed in the following two Propositions.

  \begin{proposition} \label{prop:belpha_via_psis}
    We have
    \begin{equation} \label{eq:belpha_via_psis}
      (1+\beta_p)(1+\bpsi_p)=(1+\alpha_p)(1+\apsi_p)=d/n.
    \end{equation}
  \end{proposition}

  \begin{proof}
    The parallelepiped in $\R^d$ defined by \eqref{eq:belpha_1_definition} can be written as
    \begin{equation*}
      M_\gamma(t)=\Big\{ \vec z\in\R^d \,\Big|\,
      \max_{1\leq j\leq m}|\langle\vec e_j,\vec z\rangle|\leq t,\
      \max_{1\leq i\leq n}|\langle\pmb\ell_{m+i},\vec z\rangle|\leq t^{-\gamma}
        \Big\},
    \end{equation*}
    where $\langle\,\cdot\,,\cdot\,\rangle$ denotes the inner product in $\R^d$.

    Therefore, $\beta_p$ (resp. $\alpha_p$) equals the supremum of the real numbers $\gamma$, such that there are arbitrarily large values of $t$ for which (resp. such that for every $t$ large enough) the parallelepiped $M_\gamma(t)$ contains $p$ linearly independent integer points.

    Hence, considering the parallelepipeds
    \begin{equation} \label{eq:P_gamma}
      P_\gamma(t)=T^{-1}M_\gamma(t)=\Big\{ \vec z\in\R^d \,\Big|\,
      \max_{1\leq j\leq m}|\langle\vec e_j,\vec z\rangle|\leq t,\
      \max_{1\leq i\leq n}|\langle\vec e_{m+i},\vec z\rangle|\leq t^{-\gamma}
      \Big\},
    \end{equation}
    we see that
    \begin{equation} \label{eq:belpha_p_via_parallelepipeds}
      \beta_p=\limsup_{t\to+\infty}\big\{ \gamma\in\R \,\big|\, \lambda_p(P_\gamma(t))=1 \big\}\,,\qquad\alpha_p=\liminf_{t\to+\infty}\big\{ \gamma\in\R \,\big|\, \lambda_p(P_\gamma(t))=1 \big\},
    \end{equation}
    where $\lambda_p(P_\gamma(t))$ is the $p$-th minimum of $P_\gamma(t)$ with respect to $\La$.

    But $P_{m/n}(t)=D_{\pmb\tau(\ln t)}\cB_\infty^d$, so
    \begin{equation} \label{eq:psis_via_parallelepipeds}
      \bpsi_p(\La,\gT)=\liminf_{t\to+\infty}\frac{\ln(\lambda_p(P_{m/n}(t)))}{\ln t}\,,\qquad\apsi_p(\La,\gT)=\limsup_{t\to+\infty}\frac{\ln(\lambda_p(P_{m/n}(t)))}{\ln t}\,.
    \end{equation}

    A simple calculation shows that
    \[ P_\gamma(t)=t^{\frac{m-n\gamma}{d}}P_{m/n}\big(t^{\frac{n+n\gamma}{d}}\big), \]
    i.e.
    \[ \lambda_p(P_\gamma(t))=(t')^{\frac{-m+n\gamma}{n+n\gamma}}\lambda_p\big(P_{m/n}(t')\big) \]
    with $t'=t^{\frac{n+n\gamma}{d}}$. Therefore, the equality
    \[ \lambda_p(P_\gamma(t))=1 \]
    holds if and only if
    \[ 1-\frac{d}{n+n\gamma}+\frac{\ln(\lambda_p(P_{m/n}(t')))}{\ln t'}=0. \]
    Hence, in view of \eqref{eq:belpha_p_via_parallelepipeds}, \eqref{eq:psis_via_parallelepipeds}, we get
    \[ \beta_p=\limsup_{t\to+\infty}\left\{ \frac dn\left( 1+\frac{\ln(\lambda_p(P_{m/n}(t)))}{\ln t} \right)^{-1}-1 \right\}=\frac dn\left(1+\bpsi_p\right)^{-1}-1 \]
    and
    \[ \alpha_p=\liminf_{t\to+\infty}\left\{ \frac dn\left( 1+\frac{\ln(\lambda_p(P_{m/n}(t)))}{\ln t} \right)^{-1}-1 \right\}=\frac dn\left(1+\apsi_p\right)^{-1}-1, \]
    which immediately implies \eqref{eq:belpha_via_psis}.
  \end{proof}

  \begin{proposition} \label{prop:ba_via_Psis}
    Set $\varkappa=\min(p,\frac mn(d-p))$. Then
    \begin{equation} \label{eq:ba_via_Psis}
      (1+\gb_p)(\varkappa+\bPsi_p)=(1+\ga_p)(\varkappa+\aPsi_p)=d/n.
    \end{equation}
  \end{proposition}

  \begin{proof}
    Let $\vec L_\sigma$, $\vec E_\sigma$, $\cJ_k$, $\cJ'_k$ be as in Section \ref{sec:laurexp}.

    Since $T^{-1}\pmb\ell_i=\vec e_i$ and $T^{-1}\vec e_j=\vec e_j$, if $1\leq i\leq m$ and $m+1\leq j\leq d$, we have
    \begin{equation} \label{eq:T_deapplied_to_LE}
      (T^{-1})^{(k+k')}(\vec L_\sigma\wedge\vec E_{\sigma'})=\vec E_\sigma\wedge\vec E_{\sigma'},\qquad\text{ for each }
      \sigma\in\cJ_k,\ \sigma'\in\cJ'_{k'},
    \end{equation}
    where $(T^{-1})^{(k+k')}$ is the $(k+k')$-th compound of $T^{-1}$. Furthermore, since $\La=T^{-1}\Z^d$, we have
    \begin{equation} \label{eq:T_and_La}
      \widehat\La=\wedge^p(\La)=(T^{-1})^{(p)}(\wedge^p(\Z^d)).
    \end{equation}
    Hence for each $\vec Z\in\wedge^p(\Z^d)$ and each $\sigma\in\cJ_k$, $\sigma'\in\cJ'_{d-p-k}$ (with $k$ satisfying $k_0\leq k\leq k_1$) we get
    \begin{equation} \label{eq:T_deapplied_to_LEZ}
      |\vec L_\sigma\wedge\vec E_{\sigma'}\wedge\vec Z|=
      |(T^{-1})^{(d-p)}(\vec L_\sigma\wedge\vec E_{\sigma'})\wedge(T^{-1})^{(p)}\vec Z|=
      |\vec E_\sigma\wedge\vec E_{\sigma'}\wedge\vec Z'|,
    \end{equation}
    where $\vec Z'\in\widehat\La$. Here, besides \eqref{eq:T_deapplied_to_LE}, \eqref{eq:T_and_La}, we have made use of the fact that for every $\vec V\in\wedge^p(\R^d)$, $\vec W\in\wedge^{d-p}(\R^d)$ the wedge product $\vec V\wedge\vec W$ is a real number and
    \[ |\vec V\wedge\vec W|=|T^{(p)}\vec V\wedge T^{(d-p)}\vec W|, \]
    provided $\det T=1$.

    Taking into account that any two norms in a Euclidean space are equivalent, we conclude from \eqref{eq:T_deapplied_to_LEZ} and Proposition \ref{prop:ba_substitution} that $\gb_p$ (resp. $\ga_p$) equals the supremum of the real numbers $\gamma$, such that there are arbitrarily large values of $t$ for which (resp. such that for every $t$ large enough) the system of inequalities
    \begin{equation} \label{eq:ba_modified_with_T_applied}
      \max_{\begin{subarray}{c} \sigma\in\cJ_k \\ \sigma'\in\cJ'_{d-p-k} \end{subarray}}
      |\vec E_\sigma\wedge\vec E_{\sigma'}\wedge\vec Z|\leq t^{1-(k-k_0)(1+\gamma)},\qquad k=k_0,\ldots,k_1,
    \end{equation}
    has a nonzero solution in $\vec Z\in\widehat\La$.

    The inequalities \eqref{eq:ba_modified_with_T_applied} define the parallelepiped
    \begin{equation} \label{eq:P_gamma_hat}
      \widehat P_\gamma(t)=\Big\{ \vec Z\in\wedge^p(\R^d) \,\Big|\,
      \max_{\begin{subarray}{c} \sigma\in\cJ_{m-k} \\ \sigma'\in\cJ'_{p-m+k} \end{subarray}}
      |\langle\vec E_\sigma\wedge\vec E_{\sigma'},\vec Z\rangle|\leq t^{1-(k-k_0)(1+\gamma)},\
      k=k_0,\ldots,k_1 \Big\}.
    \end{equation}
    By analogy with \eqref{eq:belpha_p_via_parallelepipeds} we can write
    \begin{equation} \label{eq:ba_via_parallelepipeds}
      \gb_p=\limsup_{t\to+\infty}\Big\{ \gamma\in\R \ \Big|\, \lambda_1\big(\widehat P_\gamma(t)\big)=1 \Big\}\,,\qquad
      \ga_p=\liminf_{t\to+\infty}\Big\{ \gamma\in\R \ \Big|\, \lambda_1\big(\widehat P_\gamma(t)\big)=1 \Big\},
    \end{equation}
    where $\lambda_1\big(\widehat P_\gamma(t)\big)$ is the first minimum of $\widehat P_\gamma(t)$ with respect to $\widehat\La$.

    Consider the path $\widehat\gT$ defined by \eqref{eq:T_hat} for $\gT$. Then
    \[ \widehat\tau_j(s)=\sum_{i\in\sigma_j}\tau_i(s), \]
    and if $\sigma_j\cap\{1,\ldots,m\}\in\cJ_{m-k}$\,, we have
    \[ \widehat\tau_j(s)=(m-k)s-\frac{(p-(m-k))m}{n}s=\left(\frac dn(k_0-k)+\varkappa\right)s=(1-(k-k_0)(1+\gamma_0))\ln t, \]
    where
    \[ t=e^{\varkappa s},\qquad\gamma_0=\frac{d}{n\varkappa}-1. \]
    Hence
    \[ \widehat P_{\gamma_0}(t)=D_{\widehat{\pmb\tau}(s)}\cB_\infty^r, \]
    where, as before, $r=\binom dp$.

    Thus, similar to \eqref{eq:psis_via_parallelepipeds}, we get
    \begin{equation} \label{eq:hat_psis_via_parallelepipeds}
      \bpsi_1(\widehat\La,\widehat\gT)=\liminf_{t\to+\infty}\frac{\varkappa\ln(\lambda_1(\widehat P_{\gamma_0}(t)))}{\ln t}\,,\qquad
      \apsi_1(\widehat\La,\widehat\gT)=\limsup_{t\to+\infty}\frac{\varkappa\ln(\lambda_1(\widehat P_{\gamma_0}(t)))}{\ln t}\,.
    \end{equation}

    The rest of the argument is very much the same as the corresponding part of the proof of Proposition \ref{prop:belpha_via_psis}. Let us observe that
    \[ \widehat P_\gamma(t)=t^{1-\frac{1+\gamma}{1+\gamma_0}}\widehat P_{\gamma_0}\big(t^{\frac{1+\gamma}{1+\gamma_0}}\big). \]
    This implies that
    \[ \lambda_1\big(\widehat P_\gamma(t)\big)=(t')^{1-\frac{1+\gamma_0}{1+\gamma}}\lambda_1\big(\widehat P_{\gamma_0}(t')\big) \]
    with $t'=t^{\frac{1+\gamma}{1+\gamma_0}}$. Therefore, the equality
    \[ \lambda_1\big(\widehat P_\gamma(t)\big)=1 \]
    holds if and only if
    \[ 1-\frac{1+\gamma_0}{1+\gamma}+\frac{\ln(\lambda_1(\widehat P_{\gamma_0}(t')))}{\ln t'}=0. \]
    Hence, in view of \eqref{eq:ba_via_parallelepipeds}, \eqref{eq:hat_psis_via_parallelepipeds}, we get
    \[ \gb_p=\limsup_{t\to+\infty}\left\{ (1+\gamma_0)\left( 1+\frac{\ln(\lambda_1(\widehat P_{\gamma_0}(t)))}{\ln t} \right)^{-1}-1 \right\}=
       (1+\gamma_0)\left(1+\varkappa^{-1}\bpsi_1(\widehat\La,\widehat\gT)\right)^{-1}-1 \]
    and
    \[ \ga_p=\liminf_{t\to+\infty}\left\{ (1+\gamma_0)\left( 1+\frac{\ln(\lambda_1(\widehat P_{\gamma_0}(t)))}{\ln t} \right)^{-1}-1 \right\}=
       (1+\gamma_0)\left(1+\varkappa^{-1}\apsi_1(\widehat\La,\widehat\gT)\right)^{-1}-1. \]
    Thus,
    \[ (1+\gb_p)(\varkappa+\bpsi_1(\widehat\La,\widehat\gT))=(1+\ga_p)(\varkappa+\apsi_1(\widehat\La,\widehat\gT))=d/n. \]
    It remains to apply Proposition \ref{prop:Psi_p_is_Psi_1}.
  \end{proof}

  \section{Transposed system}

  The subspace spanned by $\pmb\ell_{m+1},\ldots,\pmb\ell_d$ is the space of solutions to the system
  \[ -\tr\Theta\vec y=\vec x. \]
  As we noticed in Section \ref{sec:laurexp}, it coincides with the orthogonal complement $\cL^\bot$ for $\cL$. Denote by $\beta_p^\ast$, $\alpha_p^\ast$, $\gb_p^\ast$, $\ga_p^\ast$ the corresponding $p$-th regular and uniform Diophantine exponents of the first and of the second types for the matrix $\tr\Theta$. Obviously, they coincide with the ones corresponding to $-\tr\Theta$. The lattice constructed for $-\tr\Theta$ the very same way $\La$ was constructed for $\Theta$, would be
  \[ \begin{pmatrix}
       E_n & 0 \\
       \tr\Theta & E_m
     \end{pmatrix}\Z^d. \]
  But transposing the first $n$ and the last $m$ coordinates turns this lattice into
  \[ \begin{pmatrix}
       E_m & \tr\Theta \\
       0 & E_n
     \end{pmatrix}\Z^d=\tr T\Z^d=\La^\ast, \]
  which is the lattice, dual for $\La$. For this reason with $\tr\Theta$ we shall associate $\La^\ast$. Now, the most natural way to specify the path determining Schmidt's exponents associated to $\tr\Theta$ is to take into account the coordinates permutation just mentioned and consider the path $\gT^\ast:s\to\pmb\tau^\ast(s)$ defined by
  \begin{equation} \label{eq:path_ast}
    \tau^\ast_1(s)=\ldots=\tau^\ast_m(s)=-ns/m,\quad\tau^\ast_{m+1}(s)=\ldots=\tau^\ast_d(s)=s.
  \end{equation}
  Denoting
  \[ \bpsi_p^\ast=\bpsi_p(\La^\ast,\gT^\ast),\quad \apsi_p^\ast=\apsi_p(\La^\ast,\gT^\ast), \]
  \[ \bPsi_p^\ast=\bPsi_p(\La^\ast,\gT^\ast),\quad \aPsi_p^\ast=\aPsi_p(\La^\ast,\gT^\ast), \]
  we see that any statement proved for an arbitrary $\Theta$ concerning the quantities $\beta_p$, $\alpha_p$, $\bpsi_p$, $\apsi_p$, $\bPsi_p$, $\aPsi_p$ remains valid if $\Theta$ is substituted by $\tr\Theta$, and the quantities $n$, $m$, $\beta_p$, $\alpha_p$, $\bpsi_p$, $\apsi_p$ are substituted by $m$, $n$, $\beta_p^\ast$, $\alpha_p^\ast$, $\bpsi_p^\ast$, $\apsi_p^\ast$, $\bPsi_p^\ast$, $\aPsi_p^\ast$, respectively. Particularly, the analogues of Propositions \ref{prop:belpha_via_psis}, \ref{prop:ba_via_Psis} hold:

  \begin{proposition} \label{prop:starred_belpha_via_starred_psis}
    We have
    \begin{equation} \label{eq:starred_belpha_via_starred_psis}
      (1+\beta_p^\ast)(1+\bpsi_p^\ast)=(1+\alpha_p^\ast)(1+\apsi_p^\ast)=d/m.
    \end{equation}
  \end{proposition}

  \begin{proposition} \label{prop:starred_ba_via_starred_Psis}
    Set $\varkappa^\ast=\min(p,\frac nm(d-p))$. Then
    \begin{equation} \label{eq:starred_ba_via_starred_Psis}
      (1+\gb_p^\ast)(\varkappa^\ast+\bPsi_p^\ast)=(1+\ga_p^\ast)(\varkappa^\ast+\aPsi_p^\ast)=d/m.
    \end{equation}
  \end{proposition}

  Further, same as \eqref{eq:psis_via_parallelepipeds}, we get
  \begin{equation} \label{eq:starred_psis_via_parallelepipeds}
    \bpsi_p^\ast=\liminf_{t\to+\infty}\frac{\ln(\lambda_p^\ast(P_{m/n}(t^{-n/m})))}{\ln t}\,,\qquad\apsi_p^\ast=\limsup_{t\to+\infty}\frac{\ln(\lambda_p^\ast(P_{m/n}(t^{-n/m})))}{\ln t}\,,
  \end{equation}
  where $\lambda_p^\ast$ denotes the $p$-th minimum with respect to $\La^\ast$.

  Let us show that $\bpsi_p^\ast$, $\apsi_p^\ast$ are closely connected with $\bpsi_{d-p}$, $\apsi_{d-p}$ (which, as before, are related to $\La$ and the path $\gT$ defined by \eqref{eq:path}). It follows from the definition of $P_\gamma(t)$ that there is a positive constant $c$ depending only on $\Theta$, such that
  \[ c^{-1}P_\gamma(t^{-1})\subseteq P_\gamma(t)^\ast\subseteq cP_\gamma(t^{-1}), \]
  where $P_\gamma(t)^\ast$ is the polar reciprocal body for $P_\gamma(t)$. Furthermore, it follows from Mahler's theory that
  \[ \lambda_p^\ast(P_\gamma(t)^\ast)\lambda_{d+1-p}(P_\gamma(t))\asymp1 \]
  with the implied constants depending only on $d$. Hence
  \begin{equation} \label{eq:mahler_with_no_ast}
    \lambda_p^\ast(P_\gamma(t^{-1}))\lambda_{d+1-p}(P_\gamma(t))\asymp1
  \end{equation}
  Combining \eqref{eq:starred_psis_via_parallelepipeds}, \eqref{eq:mahler_with_no_ast} and \eqref{eq:psis_via_parallelepipeds} with $p$ substituted by $d+1-p$ we get

  \begin{proposition} \label{prop:starred_psis_via_psis}
    We have
    \[ \bpsi_p^\ast=-\dfrac nm\apsi_{d+1-p}\quad\text{ and }\quad\apsi_p^\ast=-\dfrac nm\bpsi_{d+1-p}\,. \]
  \end{proposition}

  \begin{corollary} \label{cor:starred_belpha_via_psis}
    We have
    \[ (1+\beta_p^\ast)(m-n\apsi_{d+1-p})=(1+\alpha_p^\ast)(m-n\bpsi_{d+1-p})=d. \]
  \end{corollary}

  \begin{proof}
    Follows from Propositions \ref{prop:starred_belpha_via_starred_psis} and \ref{prop:starred_psis_via_psis}.
  \end{proof}

  \begin{corollary} \label{cor:starred_times_nonstarred_equals_one}
    We have
    \[ \alpha_{d+1-p}\beta_p^\ast=1\quad\text{ and }\quad\alpha_{d+1-p}^\ast\beta_p=1. \]
  \end{corollary}

  \begin{proof}
    Follows from Proposition \ref{prop:belpha_via_psis} and Corollary \ref{cor:starred_belpha_via_psis}.
  \end{proof}


  In order to obtain the corresponding relations between the exponents of the second type, let us go in the opposite direction and prove

  \begin{proposition} \label{prop:starred_equals_nonstarred}
    We have
    \[ \gb_p=\gb_{d-p}^\ast\quad\text{ and }\quad\ga_p=\ga_{d-p}^\ast. \]
  \end{proposition}

  \begin{proof}
    Let $\vec L_\sigma$, $\vec E_\sigma$, $\cJ_k$, $\cJ'_k$ be as in Section \ref{sec:laurexp}.

    We remind that the bases $\pmb\ell_1,\ldots,\pmb\ell_m,\vec e_{m+1},\ldots,\vec e_d$ and $\vec e_1,\ldots,\vec e_m,\pmb\ell_{m+1},\ldots,\pmb\ell_d$ are dual. So, if $\sigma\in\cJ_k$, $\sigma'\in\cJ'_{k'}$, then
    \[ \ast(\vec L_\sigma\wedge\vec E_{\sigma'})=\pm\vec E_{\overline\sigma}\wedge\vec L_{\overline\sigma'}, \]
    where $\ast$ denotes the Hodge star operator,
    \[ \overline\sigma=\{1,\ldots,m\}\backslash\sigma,\qquad\overline\sigma'=\{m+1,\ldots,d\}\backslash\sigma', \]
    and the sign depends on the parity of the corresponding permutation. Hence for any $\sigma\in\cJ_k$, $\sigma'\in\cJ'_{d-p-k}$, and any $\vec Z\in\wedge^p(\Z^d)$ we have
    \[ |\vec L_\sigma\wedge\vec E_{\sigma'}\wedge\vec Z|=|\vec E_{\overline\sigma}\wedge\vec L_{\overline\sigma'}\wedge\ast\vec Z|. \]
    Thus,
    \begin{equation} \label{eq:max_equals_hodged_max}
      \max_{\begin{subarray}{c} \sigma\in\cJ_k \\ \sigma'\in\cJ'_{d-p-k} \end{subarray}}
      |\vec L_\sigma\wedge\vec E_{\sigma'}\wedge\vec Z|=
      \max_{\begin{subarray}{c} \sigma'\in\cJ'_{p-m+k} \\ \sigma\in\cJ_{m-k} \end{subarray}}
      |\vec L_{\sigma'}\wedge\vec E_\sigma\wedge\ast\vec Z|,
    \end{equation}
    for each $\vec Z\in\wedge^p(\Z^d)$.

    Set $k_0^\ast=\max(0,n-(d-p))$, $k_1^\ast=\min(n,p)$. Then $k_0^\ast=k_0+p-m$, $k_1^\ast=k_1+p-m$, and the inequality $k_0\leq k\leq k_1$ is equivalent to $k_0^\ast\leq p-m+k\leq k_1^\ast$. Therefore, it follows from \eqref{eq:max_equals_hodged_max} that \eqref{eq:ba_modified} is equivalent to
    \begin{equation} \label{eq:ba_hodged}
      \max_{\begin{subarray}{c} \sigma'\in\cJ'_k \\ \sigma\in\cJ_{p-k} \end{subarray}}
      |\vec L_{\sigma'}\wedge\vec E_\sigma\wedge\ast\vec Z|\leq t^{1-(k-k_0^\ast)(1+\gamma)},\qquad k=k_0^\ast,\ldots,k_1^\ast.
    \end{equation}

    It remains to apply Proposition \ref{prop:ba_substitution} and the fact that $\ast(\wedge^p(\Z^d))=\wedge^{d-p}(\Z^d)$.
  \end{proof}

  \begin{corollary} \label{cor:starred_ba_via_Psis}
    Set $\varkappa^{\ast\ast}=\min(d-p,\frac mnp)=\frac mn\varkappa^\ast$. Then
    \[ (1+\gb_p^\ast)(\varkappa^{\ast\ast}+\bPsi_{d-p})=(1+\ga_p^\ast)(\varkappa^{\ast\ast}+\aPsi_{d-p})=d/n. \]
  \end{corollary}

  \begin{proof}
    Follows from Propositions \ref{prop:ba_via_Psis} and \ref{prop:starred_equals_nonstarred}.
  \end{proof}

  \begin{corollary} \label{cor:starred_equals_nonstarred}
    We have
    \[ \bPsi_p^\ast=\dfrac nm\bPsi_{d-p}\quad\text{ and }\quad\aPsi_p^\ast=\dfrac nm\aPsi_{d-p}\,. \]
  \end{corollary}

  \begin{proof}
    Follows from Proposition \ref{prop:starred_ba_via_starred_Psis} and Corollary \ref{cor:starred_ba_via_Psis}.
  \end{proof}

  \section{Transference inequalities}

  For $p=1$ we have $\beta_1=\gb_1$, $\alpha_1=\ga_1$, which was shown in \cite{bugeaud_laurent_up_down}, or which can also be seen from our Propositions \ref{prop:belpha_via_psis}, \ref{prop:ba_via_Psis} and the obvious fact that $\bpsi_1=\bPsi_1$ and $\apsi_1=\aPsi_1$.

  In \cite{khintchine_palermo} A.~Khintchine proved for $m=1$ his famous transference inequalities
  \begin{equation} \label{eq:khintchine_transference}
    \gb_1^\ast\geq n\gb_1+n-1,\qquad
    \gb_1\geq\frac{\gb_1^\ast}{(n-1)\gb_1^\ast+n}\,.
  \end{equation}
  As we mentioned in the Introduction, M.~Laurent and Y.~Bugeaud in their paper \cite{bugeaud_laurent_up_down} split \eqref{eq:khintchine_transference} into a chain of inequalities for intermediate exponents. They proved that for $m=1$ and every $p=1,\ldots,n-1$
  \begin{equation} \label{eq:khintchine_transference_split}
    \gb_{p+1}\geq\frac{(n-p+1)\gb_{p}+1}{n-p}\,,\qquad
    \gb_{p}\geq\frac{p\gb_{p+1}}{\gb_{p+1}+p+1}\,.
  \end{equation}
  By Proposition \ref{prop:starred_equals_nonstarred} we have $\gb_1^\ast=\gb_{d-1}$. Therefore, \eqref{eq:khintchine_transference} can be easily obtained by iterating \eqref{eq:khintchine_transference_split}.

  In \cite{dyson} F.~Dyson generalized \eqref{eq:khintchine_transference} to the case of arbitrary $n$, $m$ by proving that
  \begin{equation} \label{eq:dyson_transference}
    \gb_1^\ast\geq\frac{n\gb_1+n-1}{(m-1)\gb_1+m}\,.
  \end{equation}
  It is interesting to rewrite \eqref{eq:dyson_transference} in terms of Schmidt's exponents. By Propositions \ref{prop:starred_equals_nonstarred} and \ref{prop:ba_via_Psis} it becomes simply
  \begin{equation} \label{eq:Psi_very_dyson}
    (d-1)\bPsi_{d-1}\leq\bPsi_1,
  \end{equation}
  which coincides with the first statement of Corollary \ref{cor:Psi_dyson}. But we already have an intermediate variant of this inequality! It is
  \begin{equation} \label{eq:Psi_inter_very_dyson}
    \frac{\bPsi_{p+1}}{d-p-1}\leq\frac{\bPsi_p}{d-p}\,,
  \end{equation}
  the first statement of Corollary \ref{cor:Psi_inter_dyson}. Rewriting it in terms of Diophantine exponents we get
  
  \begin{theorem} \label{t:inter_dyson}
    For each $p=1,\ldots,d-2$ the following statements hold.
    
    If $p\geq m$, then
    \begin{equation} \label{eq:inter_dyson_p_geq}
      (d-p-1)(1+\gb_{p+1})\geq(d-p)(1+\gb_p).
    \end{equation}
    
    If $p\leq m-1$, then
    \begin{equation} \label{eq:inter_dyson_p_leq}
      (d-p-1)(1+\gb_p)^{-1}\geq(d-p)(1+\gb_{p+1})^{-1}-n.
    \end{equation}
  \end{theorem}
  
  If $m=1$, then $p\geq m$ and \eqref{eq:inter_dyson_p_geq} gives the first inequality of \eqref{eq:khintchine_transference_split}. If $n=1$, then $p+1\leq m$ and \eqref{eq:inter_dyson_p_leq} in view of Proposition \ref{prop:starred_equals_nonstarred} gives the second inequality of \eqref{eq:khintchine_transference_split}.
  
  As we see, the description of the discussed phenomenon in terms of Schmidt's exponents given by \eqref{eq:Psi_inter_very_dyson} is much more elegant. Its another attraction is its universality for all values of $n$, $m$ whose sum is equal to $d$. Moreover, the second statements of Corollaries \ref{cor:Psi_inter_dyson}, \ref{cor:Psi_dyson} are the analogues of \eqref{eq:Psi_inter_very_dyson} and \eqref{eq:Psi_very_dyson} for the upper Schmidt's exponents, so rewriting them with the help of Proposition \ref{prop:ba_via_Psis} gives us the analogue of Theorem \ref{t:inter_dyson} for the uniform Diophantine exponents splitting the inequality
  \begin{equation} \label{eq:apfel_transference}
    \ga_1^\ast\geq\frac{n\ga_1+n-1}{(m-1)\ga_1+m}
  \end{equation}
  proved by A.~Apfelbeck in \cite{apfelbeck} into a chain of inequalities for intermediate exponents:

  \begin{theorem} \label{t:inter_apfel}
    For each $p=1,\ldots,d-2$ the following statements hold.

    If $p\geq m$, then
    \begin{equation} \label{eq:inter_apfel_p_geq}
      (d-p-1)(1+\ga_{p+1})\geq(d-p)(1+\ga_p).
    \end{equation}

    If $p\leq m-1$, then
    \begin{equation} \label{eq:inter_apfel_p_leq}
      (d-p-1)(1+\ga_p)^{-1}\geq(d-p)(1+\ga_{p+1})^{-1}-n.
    \end{equation}
  \end{theorem}

\vskip 10mm

\noindent
Oleg N. {\sc German} \\
Moscow Lomonosov State University \\
Vorobiovy Gory, GSP--1 \\
119991 Moscow, RUSSIA \\
\emph{E-mail}: {\fontfamily{cmtt}\selectfont german@mech.math.msu.su, german.oleg@gmail.com}

\end{document}